\documentclass[11pt]{amsart}

\usepackage{amsmath,amsthm,amssymb,tikz-cd,bm,mathdots, url}

\usepackage[margin=1in]{geometry}

\newtheorem{theorem}{Theorem}[section]					
\newtheorem{lemma}[theorem]{Lemma}
\newtheorem{corollary}[theorem]{Corollary}
\newtheorem{proposition}[theorem]{Proposition}
\theoremstyle{definition}
\newtheorem{definition}[theorem]{Definition}
\theoremstyle{definition}

\newcommand{\Z}{\mathbb{Z}}
\newcommand{\Q}{\mathbb{Q}}

\newcommand{\C}{\mathbb{C}}
\newcommand{\F}{\mathbb{F}}
\newcommand{\K}{\mathbb{K}}
							
\newcommand{\OO}{\mathcal{O}}

\newcommand{\bG}{\mathbf{G}}

\newcommand{\GL}{\operatorname{GL}}						
\newcommand{\SL}{\operatorname{SL}}

\newcommand{\PSL}{\operatorname{PSL}}
\newcommand{\POmega}{\operatorname{P\Omega}}
\newcommand{\GU}{\operatorname{GU}}
\newcommand{\SU}{\operatorname{SU}}

\newcommand{\PSU}{\operatorname{PSU}}
\newcommand{\Sp}{\operatorname{Sp}}
\newcommand{\PSp}{\operatorname{PSp}}

\newcommand{\SO}{\operatorname{SO}}
\newcommand{\Spin}{\operatorname{Spin}}
\newcommand{\type}{\operatorname}

\newcommand{\tw}[1]{{}^#1}

\newcommand{\Bl}{\operatorname{Bl}}

\newcommand{\ch}{\operatorname{char}}

\newcommand{\diag}{\operatorname{diag}}

\newcommand{\Hom}{\operatorname{Hom}}

\newcommand{\id}{\operatorname{id}}

\newcommand{\into}{\hookrightarrow}
\newcommand{\Irr}{\operatorname{Irr}}
\newcommand{\iso}{\cong}

\newcommand{\isoto}{\overset{\sim}{\to}}

\newcommand{\Lin}{\operatorname{Lin}}

\newcommand{\nor}{\trianglelefteq}

\newcommand{\onto}{\twoheadrightarrow}

\newcommand{\Res}{\operatorname{Res}}

\newcommand{\set}[1]{\left\{#1\right\}}

\newcommand{\subgp}{\leq}

\newcommand{\tr}{\operatorname{tr}}

	\makeatletter
	\newcommand{\tpitchfork}{%
		\vbox{
			\baselineskip\z@skip
			\lineskip-.52ex
			\lineskiplimit\maxdimen
			\m@th
			\ialign{##\crcr\hidewidth\smash{$-$}\hidewidth\crcr$\pitchfork$\crcr}
		}%
	}
	\makeatother

\usepackage{hyperref}
\title{Real $2$-blocks in quasi-simple groups}

\author[J. R. McHugh]{John Revere McHugh}
\address[McHugh]{Department of Mathematics, University of Denver, Denver, CO 80210, USA}
\email{John.R.McHugh@du.edu}
\author[A. A. Schaeffer Fry]{A. A. Schaeffer Fry}
\address[Schaeffer Fry]{Department of Mathematics, University of Denver, Denver, CO 80210, USA}
\email{mandi.schaefferfry@du.edu}

\thanks{We thank John Murray for suggesting this question and for valuable discussions on the topic. The second author is grateful for support from the U.S. National Science Foundation, Award No. DMS-2439897}

\begin{document}

\begin{abstract}
   We determine which quasi-simple groups have a non-principal $2$-block that is stable under complex conjugation. As a corollary, we determine that the Mathieu group $M_{22}$ is the only simple group not possessing a nontrivial irreducible Brauer character of quadratic type, answering a recent question of Gow and Murray.
\end{abstract}

\maketitle
\section{Introduction}

The study of reality properties of representations and characters of finite groups has been a prolific area of research since the early years of character theory. The characters and Brauer characters corresponding to irreducible self-dual modules, which we will call self-dual (Brauer) characters, are exactly those irreducible (Brauer) characters that are stable under complex conjugation, relating this study to that of self-dual modules.

Motivated by the computation of decomposition numbers, Fong showed in \cite[Lem.~1]{fong74} that any nontrivial irreducible self-dual $FG$-module for a finite group $G$ and perfect field $F$ of characteristic $2$ necessarily admits a nondegenerate alternating bilinear form. This began the study of whether, further, for such a $V$ there is a nondegenerate $G$-invariant quadratic form. If $V$ is equipped with a nondegenerate $G$-invariant quadratic form, $V$ is said to be of \emph{quadratic type}. 
Modules and Brauer characters of quadratic type have been highly studied, for example in \cite{GowWillems93, GowWillems95, gowwillems97, murraynavarro, GowMurray}. 

This paper is motivated by a question of Gow and Murray on the topic. Namely, in  \cite[Sec. 5]{GowMurray}, Gow and Murray ask whether each nonabelian simple group aside from the Mathieu group $M_{22}$ has some nontrivial, quadratic-type Brauer character. 
By \cite[Prop 1.4]{GowWillems93}, any real $2$-block $B$ (that is, one stable under complex conjugation - see \ref{sec:blocks} below) of a finite group contains a self-dual Brauer character $\psi$. Then  by \cite[Lem.~17]{GowMurray}, $\psi$ is of quadratic type unless possibly if $B$ is the principal block. Hence if $G$ has a non-principal real $2$-block, then there exists a nontrivial Brauer character of quadratic type. 

This naturally leads to the question of which groups possess a non-principal real $2$-block. 
We answer this question for nonabelian (quasi)-simple groups:

\begin{theorem}\label{thm:real2blocks}
    Let $G$ be a finite quasi-simple group. Then $G$ has a non-principal real $2$-block if and only if $G/Z(G)$ is not one of $M_{11}$, $M_{22}$, $M_{23}$,  $M_{24}$, $\PSL_3(3)$, or $\PSU_3(3)$.
\end{theorem}

In \cite{GowMurray}, it is shown that the sporadic simple groups noted in Theorem \ref{thm:real2blocks} contain an irreducible Brauer character of quadratic type except for $M_{22}$.
As a corollary, we may answer the question of Gow and Murray:
\begin{corollary}\label{cor:quadraticselfdual}
    Let $G$ be a finite nonabelian simple group. Then $G$ has a nontrivial irreducible Brauer character of quadratic type if and only if $G$ is not $M_{22}$.
\end{corollary}

Some of the key results we use are those of Fong--Srinivasan \cite{FS82} describing the blocks of general linear and unitary groups, which were extended in \cite{Broue_1986}, as well as those of Srinivasan--Vinroot \cite{Srinivasan_2015, SV20} on the compatibility of Lusztig's  Jordan decomposition of characters with respect to Galois action, which were extended in \cite{Schaeffer_Fry_2025}. 

The paper is structured as follows. In Section \ref{sec:preliminaries}, we introduce several preliminary notions on characters and blocks of finite groups and of groups of Lie type that will be useful for the remainder of the paper. In Section \ref{sec:othergroups}, we discuss sporadic groups, alternating groups, and groups of Lie type defined in characteristic $2$, before moving in Section \ref{sec:Lietype} to our main considerations: groups of Lie type defined in odd characteristic. Finally, the proofs of Theorem \ref{thm:real2blocks} and Corollary \ref{cor:quadraticselfdual} are completed in Section \ref{sec:mainproofs}.

\section{Preliminaries}\label{sec:preliminaries}

We begin with some preliminaries that will be useful throughout. 

\subsection{Group-Theoretic Notation}

	Let $G$ be a group. If $g,h\in G$ then we set ${}^gh:=ghg^{-1}$, and if $H\subgp G$ is a subgroup of $G$ then ${}^gH:=gHg^{-1}$. We write $H\nor G$ if $H$ is a normal subgroup of $G$. The center of $G$ is denoted $Z(G)$, and the derived subgroup is denoted $G_{\text{der}}$. The normalizer (respectively, centralizer) of a subgroup $H$ of $G$ is denoted $N_G(H)$ (resp., $C_G(H)$).
	
	Let $g\in G$ be an element of finite order. If $\ell$ is a prime number, then $g$ is an $\ell'$-\textit{element} if $(|g|,\ell)=1$, i.e., if the order of $g$ is prime to $\ell$. If $|g|$ is a power of $\ell$ then $g$ is an $\ell$-\textit{element}. So, for example, the identity of $G$ is the only element of $G$ which is both an $\ell$-element and an $\ell'$-element. An arbitrary element $g\in G$ of finite order can be expressed uniquely as a product of commuting elements $g=g_{\ell}g_{\ell'}=g_{\ell'}g_{\ell}$ such that $g_{\ell}\in G$ is an $\ell$-element and $g_{\ell'}\in G$ is an $\ell'$-element. 

    Recall that an element $g\in G$ is \textit{real} if $g$ is $G$-conjugate to $g^{-1}$.

\subsection{Characters and Blocks}
We next recall some generalities on the character and block theories of finite groups. Let $G$ be a finite group, let $\ell$ be a prime number, and let $\OO$ be a complete discrete valuation ring whose field of fractions $\K$ has characteristic 0 and whose residue field $\OO/J(\OO)$ has characteristic $\ell$. Assume that $\K$ is ``large enough'' for $G$, which occurs e.g. when $\K$ contains a root of unity of order equal to the exponent of $G$. Also, let $\K_0$ denote the algebraic closure of $\Q$ in $\K$ and set $\OO_0:=\K_0\cap\OO$.

\subsubsection{Irreducible Characters}\label{irredchars}
	Let $\Irr_{\K}(G)$ denote the set of irreducible $\K$-characters of $G$. If $\chi\in\Irr_{\K}(G)$ then $\chi$ takes values in $\K_0$. After fixing an identification of $\K_0$ with a subfield of the complex numbers, we may thus view $\chi$ as a $\C$-character of $G$. Because $\K$ is large enough, in this way we can and will identify $\Irr_{\K}(G)$ with the set $\Irr(G)$ of irreducible (complex) characters of $G$.
	
	We use standard character-theoretic notation throughout this note. The principal character of $G$ is denoted $1_G$. The group of linear characters of $G$, which are the characters $\lambda\in\Irr(G)$ satisfying $\lambda(1)=1$, is denoted $\Lin(G)$. Recall that $\Lin(G)$ acts on $\Irr(G)$ by pointwise multiplication. If $H\subgp G$ and $\chi$ is a character of $G$, then $\Res_H^G(\chi)$ denotes the restriction of $\chi$ to $H$. If $g\in G$ and $\theta$ is a character of $H$, then ${}^g\theta$ denotes the character of ${}^gH$ defined by ${}^g\theta({}^gh)=\theta(h)$ for all $h\in H$.

\subsubsection{Real Irreducible Characters}
	If $\xi\in\C$, then $\overline{\xi}$ denotes the complex conjugate of $\xi$. We may assume without loss that the subfield $\K_0\subseteq\C$ is closed under complex conjugation. Then the \textit{complex conjugate} of an irreducible character $\chi$ of $G$ is the irreducible character $\overline{\chi}$ defined by $\overline{\chi}(g)=\overline{\chi(g)}$ for all $g\in G$. It is well-known that $\overline{\chi}(g)=\chi(g^{-1})$ for any $\chi\in\Irr(G)$ and any $g\in G$. 
	
	A character $\chi\in\Irr(G)$ is called \textit{real} if it satisfies $\chi=\overline{\chi}$, which occurs if and only if $\chi$ is real-valued. We remark that if $M$ is an irreducible $\K G$-module affording the character $\chi$, then $\overline{\chi}$ is afforded by the dual $\K G$-module $\Hom_{\K}(M,\K)$. For this reason, real characters of $G$ are also known in the literature as \textit{self-dual} characters.

\subsubsection{Blocks}\label{sec:blocks}
    The primitive idempotents of the center $Z(\OO G)$ of the group algebra $\OO G$ are called the $\ell$\textit{-blocks} (or simply the \textit{blocks}) of $\OO G$. The set $\Bl(\OO G)$ of blocks of $\OO G$ is finite, and we have \[1_{\OO G}=\sum_{b\in\Bl(\OO G)}b.\]

    Let $\chi\in\Irr_{\K}(G)$. By abuse of notation, the $\K$-linear extension of $\chi$ to a function defined on all of $\K G$ is again denoted $\chi$. If $b\in\Bl(\OO G)$, then $\chi$ \textit{belongs to} $b$ if $\chi(gb)=\chi(g)$ for all $g\in G$. The set of irreducible characters that belong to $b$ is denoted $\Irr(b)$. Recall that the subsets $\Irr(b)$, $b\in\Bl(\OO G)$, form a partition of $\Irr_{\K}(G)$. The unique block $b_0$ to which the principal character $1_G$ belongs is called the \textit{principal block} of $\OO G$.
     
	It is well-known (cf. \cite[Cor.~6.5.5]{Linckelmann_2018_2}) that any block of $\OO G$ can be expressed as an $\OO_0$-linear combination of the elements of $G$. Thus, if $b=\sum_{x\in G}\lambda_xx$ is a block of $\OO G$ (where $\lambda_x\in\OO_0$ for each $x\in G$) then applying complex conjugation to each of the coefficients $\lambda_x$ yields a block $\overline{b}:=\sum_{x\in G}\overline{\lambda_x}x$ of $\OO G$. The map $b\mapsto\overline{b}$ therefore defines a permutation of $\Bl(\OO G)$. A block $b$ of $\OO G$ is called \textit{real} if $b=\overline{b}$.
	
	If $b$ is a block of $\OO G$ and $\chi\in\Irr(b)$, then the complex conjugate character $\overline{\chi}$ belongs to the block $\overline{b}$. Complex conjugation thus defines a permutation of the subsets $\Irr(b)$, $b\in\Bl(\OO G)$. The following lemma is immediate from the definitions above.

\begin{lemma}\label{lem:realblocksfingrps}
	Let $b$ be a block of $\OO G$. Then $b$ is real if and only if $\Irr(b)$ is stable under taking complex conjugates. In particular, if there exists a real-valued character $\chi\in\Irr(b)$, then $b$ is real. Since the principal character $1_G$ is real-valued, the principal block $b_0$ of $\OO G$ is real.
\end{lemma}

\subsubsection{Block Covering}
	Let $N\nor G$, $b\in\Bl(\OO G)$, and $c\in\Bl(\OO N)$. Recall that $b$ is said to \textit{cover} the block $c$ if there exists an irreducible character $\chi\in\Irr(b)$ such that $\Res_N^G(\chi)$ has an irreducible constituent belonging to $c$. In fact, a stronger statement holds in this case: if $b$ covers $c$, then $\Res_N^G(\chi)$ has an irreducible constituent belonging to $c$ for any $\chi\in\Irr(b)$ (cf. \cite[Lem.~ 5.5.7]{Nagao_1989}). For example, the principal block $b_0$ of $\OO G$ always covers the principal block $c_0$ of $\OO N$. However, there may be other blocks of $\OO G$ that cover the principal block of $\OO N$. The lemma below identifies a situation in which we can say precisely which blocks of $\OO G$ have this property.

\begin{lemma}\label{lem:blockscoveringprincderived}
	Let $N\nor G$ and assume that $G/N$ is abelian. Let $b, b'\in\Bl(\OO G)$ and $c\in\Bl(\OO N)$, and assume that $b'$ covers $c$. Then $b$ covers $c$ if and only if $\Irr(b)=\Irr(b')\cdot\lambda$ for some $\lambda\in\Lin(G)$ with $N$ in its kernel. In particular, if $c=c_0$ is the principal block of $\OO N$, then $b$ covers $c_0$ if and only if there exists a linear character $\lambda\in\Lin(G)$ such that $\lambda\in\Irr(b)$ and $N\subgp\ker(\lambda)$. Hence,
    if $N=G_{\text{der}}$ is the derived subgroup of $G$, then $b$ covers $c_0$ if and only if there exists a linear character $\lambda\in\Lin(G)$ such that $\lambda\in\Irr(b)$.
\end{lemma}

\begin{proof}
The first statement follows from Clifford theory (see for example \cite[Lem.~2.2(a)]{KM13}), and the second follows by taking $b'$ to be the principal block of $\OO G$.
    The final statement holds because the derived subgroup of $G$ is contained in $\ker(\lambda)$ for all $\lambda\in\Lin(G)$. 
\end{proof}

\subsubsection{Block Domination}

Given a normal subgroup $N\nor G$ and a block $c\in\Bl(\OO [G/N]),$ there is a unique block $b\in\Bl(\OO G)$ such that $\Irr(c)\subseteq\Irr(b)$, where characters of $G/N$ are viewed as characters of $G$ via inflation. We say such a block $b$ \emph{dominates} the block $c$. In general, $b$ may dominate multiple blocks of $\OO[G/N]$. However, in certain situations, we see that this cannot happen:

\begin{lemma}\label{lem:domination}
Let $N\lhd G$ and assume that $p\nmid |N|$ or that $N\leq Z(G)$. If $b\in\Bl(\OO G)$ and $N\subgp\ker\chi$ for some $\chi\in\Irr(b)$ then $b$ dominates a unique block of $\OO[G/N]$. 
\end{lemma}
\begin{proof}
If $|N|$  is not divisible by $p$, this follows from \cite[Thm.~9.9(c)]{Nav98}. If $N$ is central, it follows from \cite[Lem.~17.2]{CE04}.
\end{proof}

\subsection{Groups of Lie type}\label{sec:prelimLie}
We now recall some well-known results on finite groups of Lie type and their irreducible characters. Let $p$ be a prime number (potentially equal to $\ell$), let $k$ denote the algebraic closure of the finite field $\F_p$, let $\mathbf{G}$ be a connected reductive algebraic group defined over $k$, and let $F:\mathbf{G}\to\mathbf{G}$ be a Steinberg endomorphism of $\mathbf{G}$. The group of $F$-fixed points of $\mathbf{G}$ is a finite group, denoted $\mathbf{G}^F$.

\subsubsection{Tori and Weyl Groups}\label{toritype}(cf. \cite[Section 3.3]{Carter_1985}, \cite[1.6.4]{Geck_2020})
	Let $\mathbf{T}_0\subgp\mathbf{G}$ be an $F$-stable maximal torus that is maximally $F$-split, i.e., contained in an $F$-stable Borel subgroup of $\mathbf{G}$. Let $\mathbf{W}=N_{\mathbf{G}}(\mathbf{T}_0)/\mathbf{T}_0$ denote the corresponding Weyl group and let $\sigma$ denote the automorphism of $\mathbf{W}$ induced by $F$. Recall that elements $u,w\in\mathbf{W}$ are said to be $\sigma$\textit{-conjugate} if there exists an element $v\in\mathbf{W}$ such that $w=vu\sigma(v)^{-1}$. 
	
	Let $\mathbf{T}\subgp\mathbf{G}$ be an $F$-stable maximal torus. Then there exists an element $g\in\mathbf{G}$ such that ${}^g\mathbf{T}_0=\mathbf{T}$. We have $g^{-1}F(g)\in N_{\mathbf{G}}(\mathbf{T}_0)$, whence $g^{-1}F(g)\mathbf{T}_0\in\mathbf{W}$. The association $\mathbf{T}\mapsto g^{-1}F(g)\mathbf{T}_0$ descends to a well-defined bijection between the sets of $\mathbf{G}^F$-conjugacy classes of $F$-stable maximal tori of $\mathbf{G}$ and $\sigma$-conjugacy classes of $\mathbf{W}$. Furthermore, if the class of an $F$-stable maximal torus $\mathbf{T}$ corresponds to the $\sigma$-conjugacy class of an element $w\in\mathbf{W}$ then $\mathbf{T}^F={}^g\mathbf{T}_0[w]$, where
	\begin{equation*}
		\mathbf{T}_0[w]=\set{t\in\mathbf{T}_0|F(t)=\dot{w}^{-1}t\dot{w}},
	\end{equation*}
	$\dot{w}\in N_{\mathbf{G}}(\mathbf{T}_0)$ is any element that maps to $w$ under the canonical projection $N_{\mathbf{G}}(\mathbf{T}_0)\onto\mathbf{W}$, and $g\in\mathbf{G}$ is any element satisfying ${}^g\mathbf{T}_0=\mathbf{T}$ and $g^{-1}F(g)\mathbf{T}_0=w$.

\subsubsection{More on Semisimple Elements}
The following lemma will be useful in identifying real elements of $\mathbf{G}^F$. Its proof is well-known; see, e.g., \cite[Example 1.4.10]{Geck_2020}.

\begin{lemma}\label{lem:realelts}
	Let $\mathbf{G}$ be a connected reductive algebraic group and let $F:\mathbf{G}\to\mathbf{G}$ be a Steinberg endomorphism of $\mathbf{G}$. Let $s\in\mathbf{G}^F$ be such that $C_{\mathbf{G}}(s)$ is connected. Then $s$ is $\mathbf{G}$-conjugate to $s^{-1}$ if and only if $s$ is $\mathbf{G}^F$-conjugate to $s^{-1}$ (i.e., $s$ is a real element of $\mathbf{G}^F$).
\end{lemma}

The next lemma is also well-known (cf. for example \cite[Exercise 20.16]{Malle_Testerman_2011}) and can sometimes be applied to show that the centralizer of a semisimple element is connected.

\begin{lemma}\label{lem:centconn}
	Let $\mathbf{G}$ be a semisimple algebraic group over $k$ and let $\pi:\mathbf{G}_{\text{sc}}\onto\mathbf{G}$ be the natural isogeny from a simply connected semisimple group with the same root system as $\mathbf{G}$. If $s\in\mathbf{G}$ is a semisimple element whose order is relatively prime to $|\ker\pi|$, then $C_{\mathbf{G}}(s)$ is connected.
\end{lemma}

\subsubsection{Lusztig Series}\label{Lusztigseries}(cf. \cite{Geck_2020})
    If $\mathbf{T}$ is an $F$-stable maximal torus of $\mathbf{G}$ and $\theta$ is an irreducible character of $\mathbf{T}^F$, then the corresponding Deligne-Lusztig character of $\mathbf{G}^F$ is denoted $R_{\mathbf{T}}^{\mathbf{G}}(\theta)$. The Deligne-Lusztig characters $R_{\mathbf{T}}^{\mathbf{G}}(\theta)$ are in fact \textit{virtual} characters of $\mathbf{G}^F$, i.e., they are contained in the $\Z$-span of $\Irr(\mathbf{G}^F)$ within the space of class functions on $\mathbf{G}^F$. 	
	
    Let $\mathbf{G}^\ast$ be another connected reductive algebraic group defined over $k$ and let $F^\ast:\mathbf{G}^\ast\to\mathbf{G}^\ast$ be a Steinberg endomorphism of $\mathbf{G}^\ast$ such that the pair $(\mathbf{G}^\ast,F^\ast)$ is \textit{dual} to $(\mathbf{G},F)$. In other words, assume there exists a maximally $F$-split torus $\mathbf{T}_0\subgp\mathbf{G}$, a maximally $F^\ast$-split torus $\mathbf{T}_0^\ast\subgp\mathbf{G}^\ast$, and an isomorphism $\delta:X(\mathbf{T}_0)\isoto Y(\mathbf{T}_0^\ast)$ between the character group of $\mathbf{T}_0$ and the cocharacter group of $\mathbf{T}_0^\ast$ such that $\delta$ intertwines the actions of $F$ and $F^\ast$ and defines an isomorphism of root data (cf. \cite[Def.~1.5.17]{Geck_2020}).
	
	We fix, once and for all, group isomorphisms $k^\times\isoto(\Q/\Z)_{p'}$ and $(\Q/\Z)_{p'}\isoto\mu_{p'}$, where $\mu_{p'}$ denotes the group of complex roots of unity that have order prime to $p$ (note that composing these two isomorphisms defines an injective group homomorphism $k^\times\into\C^\times$). Once these isomorphisms are fixed, one obtains
	\begin{itemize}
		\item[(1)] a group isomorphism $\mathbf{T}_0^{\ast F^\ast}\isoto\Irr(\mathbf{T}_0^F)$, which is determined by $\delta$ and is denoted $s\mapsto\hat{s}$; and 
		\item[(2)] a partition of $\Irr(\mathbf{G}^F)$ into subsets $\mathcal{E}(\mathbf{G}^F,s)$, called \textit{(rational) Lusztig series}, where $s$ is a semisimple element of $\mathbf{G}^{\ast F^\ast}$.
	\end{itemize}
	Two Lusztig series $\mathcal{E}(\mathbf{G}^F,s)$ and $\mathcal{E}(\mathbf{G}^F,t)$ are equal if and only if the semisimple elements $s,t\in\mathbf{G}^{\ast F^\ast}$ are $\mathbf{G}^{\ast F^\ast}$-conjugate. The irreducible characters that belong to $\mathcal{E}(\mathbf{G}^F,1)$ are the \textit{unipotent characters} of $\mathbf{G}^F$. For example, the principal character of $\mathbf{G}^F$ is a unipotent character.
    
	For the following, we refer the reader to \cite[Props.~2.5.20, 2.5.21]{Geck_2020}. Let $z\in Z(\mathbf{G}^{\ast})^{F^\ast}$. Note that $z\in\mathbf{T}_0^{\ast F^\ast}$. Thus, by (1) above, there is a corresponding linear character $\hat{z}\in\Irr(\mathbf{T}_0^F)$. There exists a unique linear character of $\mathbf{G}^F$ that extends $\hat{z}$ and contains all unipotent elements of $\mathbf{G}^F$ in its kernel. We denote this irreducible character of $\mathbf{G}^F$, abusively, by $\hat{z}$. Then $\hat{z}\in\mathcal{E}(\mathbf{G}^F,z)$, and for any semisimple element $s\in\mathbf{G}^{\ast F^\ast}$ we have $$\mathcal{E}(\mathbf{G}^F,s)\cdot\hat{z}=\mathcal{E}(\mathbf{G}^F,sz).$$
    
     The permutation of $\Irr(\mathbf{G}^F)$ induced by complex conjugation descends to a permutation of the Lusztig series. To be precise, for any semisimple element $s\in\mathbf{G}^{\ast F^\ast}$ we have $$\overline{\mathcal{E}(\mathbf{G}^F,s)}=\mathcal{E}(\mathbf{G}^F,s^{-1}).$$ (See, for example, \cite[Prop.~3.3.15]{Geck_2020}.)
 
\subsubsection{The Jordan Decomposition of Characters} \label{Jdecomp}

Given a semisimple element $s\in{\bG^\ast}^{F^\ast}$, there is a bijection 
\begin{equation*}
		J_s:\mathcal{E}(\mathbf{G}^F,s)\isoto\mathcal{E}(C_{\mathbf{G}^\ast}(s)^{F^\ast},1),
\end{equation*}
and a collection of bijections as above is called a \textit{Jordan decomposition} if it satisfies the inner product condition in \cite[Thm.~2.6.4]{Geck_2020}. Such a collection exists by the work of Lusztig \cite{Lusztig_1988} (see \cite[Thm.~2.6.4,  Rem.~2.6.26]{Geck_2020}). Here we will only need to consider the case that $s\in\mathbf{G}^{\ast F^\ast}$ satisfies $C_{\mathbf{G}^\ast}(s)^{F^\ast}\subgp C_{\mathbf{G}^\ast}^\circ(s)$, where $C_{\mathbf{G}^\ast}^\circ(s)$ denotes the connected component of the identity in $C_{\mathbf{G}^\ast}(s)$. In this situation, there exists by \cite{Digne_1990, Schaeffer_Fry_2025} a unique Jordan decomposition satisfying certain additional properties. (We will therefore refer in this case to ``the" Jordan decomposition of characters.)

	One key property of the Jordan decomposition in this situation is its equivariance with respect to Galois action. The next result follows from  \cite{Srinivasan_2015} when $Z(\bG)$ is connected and from \cite[Thms.~1.1,~2.1]{Schaeffer_Fry_2025} when $Z(\bG)$ is not necessarily connected but  $C_{\mathbf{G}^\ast}(s)^{F^\ast}\subgp C_{\mathbf{G}^\ast}^\circ(s)$.

\begin{lemma}\label{lem:Jdecompprops} 
	Assume that $F$ is a Frobenius endomorphism of $\mathbf{G}$. Let $s\in\mathbf{G}^{\ast F^\ast}$ be a semisimple element satisfying $C_{\mathbf{G}^\ast}(s)^{F^\ast}\subgp C_{\mathbf{G}^\ast}^\circ(s)$ and let $\chi\in\mathcal{E}(\mathbf{G}^F,s)$. Then the Jordan decomposition of $\mathbf{G}^F$ satisfies $J_{s^{-1}}(\overline{\chi})=\overline{J_s(\chi)}$. In particular, if $\nu=J_s(\chi)$, then $\chi$ is real if and only if there exists an element $g\in \mathbf{G}^{\ast F^\ast}$ such that ${}^gs=s^{-1}$ and ${}^g\nu=\overline{\nu}$.
\end{lemma}

\subsubsection{Semisimple Characters}\label{sschars}
	Let $s\in\mathbf{G}^{\ast F^\ast}$ be such that $C_{\mathbf{G}^\ast}(s)^{F^\ast}\subgp C_{\mathbf{G}^\ast}^\circ(s)$. Then the unique irreducible character in $\mathcal{E}(\mathbf{G}^F,s)$ that corresponds via the Jordan decomposition $J_s$ to the principal character of $C_{\mathbf{G}^\ast}(s)^{F^\ast}$ is known as a \textit{semisimple} character of $\mathbf{G}^F$. In what follows, we will often write $\chi_s$ for the unique semisimple character that belongs to $\mathcal{E}(\mathbf{G}^F,s)$ in this situation.

\subsubsection{Blocks of $\bG^F$}\label{blocksGF} (cf. \cite{Broue_1989})
	Assume that the prime $\ell$ does not equal $p=\ch(k)$ and that $F$ is a Frobenius endomorphism of $\mathbf{G}$. Let $s$ be a semisimple $\ell'$-element of $\mathbf{G}^{\ast F^\ast}$. Define
	\begin{equation*}
		\mathcal{E}_\ell(\mathbf{G}^F,s):=\bigcup_{x\in C_{\mathbf{G}^{\ast}}(s)_\ell^{F^\ast}}\mathcal{E}(\mathbf{G}^F,sx),
	\end{equation*}
	where $C_{\mathbf{G}^{\ast}}(s)_\ell^{F^\ast}$ denotes the set of $\ell$-elements of $C_{\mathbf{G}^{\ast}}(s)^{F^\ast}$. Then there exists a subset $\mathcal{S}\subseteq\Bl(\OO[\mathbf{G}^F])$ such that
	\begin{equation*}
		\mathcal{E}_\ell(\mathbf{G}^F,s)=\bigcup_{b\in\mathcal{S}}\Irr(b).
	\end{equation*}
	In other words, $\mathcal{E}_\ell(\mathbf{G}^F,s)$ is a ``union of $\ell$-blocks'' of $\OO[\mathbf{G}^F]$ (\cite[Thm.~ 2.2]{Broue_1989}).
	
     Now let $s$ and $t$ be semisimple $\ell'$-elements of $\mathbf{G}^{\ast F^\ast}$. Then, as for Lusztig series, one has $\mathcal{E}_{\ell}(\mathbf{G}^F,s)=\mathcal{E}_\ell(\mathbf{G}^F,t)$ if and only if $s$ is $\mathbf{G}^{\ast F^\ast}$-conjugate to $t$. Moreover, $\Irr(\mathbf{G}^F)$ is equal to the disjoint union of the subsets $\mathcal{E}_\ell(\mathbf{G}^F,s)$ where $s$ runs over a set of representatives for the conjugacy classes of semisimple $\ell'$-elements of $\mathbf{G}^{\ast F^\ast}$. It follows that if $b\in\Bl(\OO[\mathbf{G}^F])$ then there exists a semisimple $\ell'$-element $s\in\mathbf{G}^{\ast F^\ast}$, unique up to conjugation, such that $\Irr(b)\subseteq\mathcal{E}_\ell(\mathbf{G}^F,s)$. For example, if $b_0$ denotes the principal block of $\OO[\mathbf{G}^F]$ then $\Irr(b_0)\subseteq\mathcal{E}_\ell(\mathbf{G}^F,1)$.

\section{Initial Results}\label{sec:othergroups}

\begin{lemma}\label{lem:sporalt}
    Let $G$ be a sporadic simple group, an alternating group, or the Tits group $\tw{2}\type{F}_4(2)'$. Then $\OO G$ possesses a non-principal real $2$-block if and only if $G$ is not one of $M_{11}, M_{22}, M_{23}, M_{24}$.
\end{lemma}

\begin{proof}
    If $G$ is a sporadic group,  $\tw{2}\type{F}_4(2)'$, or an alternating group $A_n$ with $5\leq n\leq 7$, the statement can be seen using GAP \cite{GAP}; this was already noted for sporadic groups in \cite[Sec.~5]{GowMurray}. 
    
    Now let $G=A_n$ be an alternating group with $n\geq8$. A character of the symmetric group $S_n$ lies in the principal block if and only if it corresponds to a partition with empty $2$-core for $n$ even, resp. $2$-core $(1)$ if $n$ is odd. Then the character corresponding to the partition $(1,2,n-3)$ lies in a non-principal $2$-block of $S_n$ if $n$ is even and the character corresponding to $(1,n-1)$ lies in a non-principal $2$-block of $S_n$ if $n$ is odd. Further, in each case the partition is not self-conjugate (recall here that $n\geq 8$), so restricts irreducibly to a (necessarily real-valued) character $\chi\in\Irr(A_n)$. As $[S_n:A_n]=2$, the principal $2$-block of $S_n$ is the unique block above the principal block of $A_n$, and hence $\chi$ lies in a real, non-principal $2$-block of $A_n$.
\end{proof}
   
\begin{lemma}\label{lem:defcharRee}
    Keep the notation in Section \ref{sec:prelimLie}, and suppose that either the characteristic of $k$ is $2$ or that $\bG^F=\tw{2}\type{G}_2(3^{2n+1})$ with $n\geq 2$. Then $\OO[\bG^F]$ and $\OO[\bG^F/Z(\bG^F)]$ each possess a non-principal real $2$-block.
\end{lemma}
\begin{proof}
    If the characteristic of $k$ is $2$, then the Steinberg character of $G=\bG^F$ lies in a block of defect zero, is real-valued, and is trivial on the center (see e.g. \cite[Prop.~3.4.10]{Geck_2020}), yielding a non-principal real $2$-block.

    If $\bG^F=\tw{2}\type{G}_2(3^{2n+1})$ with $n\geq 2$, write $q^2=3^{2n+1}$. Then from \cite{luebeckwebsite}, we see there are $\frac{1}{6}(q^2+\sqrt{3}q)$ characters of degree $(q^2+1-q\sqrt{3})(q^4-1)$ and $\frac{1}{6}(q^2-\sqrt{3}q)$ characters of degree $(q^2+1+q\sqrt{3})(q^4-1)$. Note that each of these is defect-zero, and that one of these multiplicities must be odd. Hence, there must be a real-valued defect-zero character, and hence a non-principal real $2$-block.
\end{proof}

\section{ Groups of Lie type in Odd characteristic}\label{sec:Lietype}

We keep all the notation set in Section \ref{sec:preliminaries}. In addition, we now set $\ell=2$. Our goal is to determine which finite simple groups of Lie type have no non-principal real $2$-block (although, we note that our analysis in this section will cover other non-simple finite groups of Lie type as well). Thanks to Lemma \ref{lem:defcharRee}, we only need to consider those finite simple groups that arise from connected reductive algebraic groups defined over fields of odd characteristic and such that $F$ is a Frobenius endomorphism. Henceforth, all connected reductive groups considered will be defined over $k=\bar{\F}_p$ for some odd prime $p$. 

Our method for proving the existence of non-principal real $2$-blocks in a finite group of Lie type will be to demonstrate that the dual group possesses a nontrivial semisimple element satisfying the conditions in the next definition.

\begin{definition}\label{defn:conditions}
	Let $\mathbf{G}$ be a connected reductive algebraic group over $k=\bar{\F}_p$ with $p\neq 2$ and let $F:\mathbf{G}\to\mathbf{G}$ be a Steinberg endomorphism of $\mathbf{G}$. Let $1\neq s\in\mathbf{G}^F$ be semisimple. We say $s$ satisfies condition (\ref{condA}) if
	\begin{equation}\tag{$\mathcal{A}$}\label{condA}
		s\text{ is real, has odd order, and }C_{\mathbf{G}}(s)^F\subgp C_{\mathbf{G}}^\circ(s).
	\end{equation} 
	We say $s$ satisfies condition (\ref{condB}) if
	\begin{equation}\tag{$\mathcal{B}$}\label{condB}
		s\text{ is not }\mathbf{G}^F\text{-conjugate to any element of the form }sz\text{, }1\neq z\in Z(\mathbf{G}^F).
	\end{equation}
	Finally, we say $s$ satisfies condition (\ref{condC}) if
	\begin{equation}\tag{$\mathcal{C}$}\label{condC}
		s\in(\mathbf{G}^F)_{\text{der}}.
	\end{equation}
\end{definition}

With these conditions defined, we may use them to prove sufficient conditions for groups of Lie type to have a non-principal real $2$-block.

\begin{proposition}\label{prop:main}
	Let $\mathbf{G}$ be a connected reductive algebraic group defined over $k=\bar{\F}_p$ with $p\neq 2$. Let $F:\mathbf{G}\to\mathbf{G}$ be a Frobenius endomorphism of $\mathbf{G}$, and let $(\mathbf{G}^\ast,F^\ast)$ be dual to the pair $(\mathbf{G},F)$. Set $G=\mathbf{G}^F$ and $G^\ast=\mathbf{G}^{\ast F^\ast}$. 
	\begin{itemize}
		\item[(a)] Suppose that $G^\ast$ contains a nontrivial semisimple element $s$ that satisfies condition (\ref{condA}). Then $\OO G$ possesses a non-principal real $2$-block. 
		\item[(b)] Suppose that $G/G_{\text{der}}$ is a $p'$-group. If $G^\ast$ contains a nontrivial semisimple element $s$ that satisfies both conditions (\ref{condA}) and (\ref{condB}), then $\OO[G_{\text{der}}]$ possesses a non-principal real $2$-block.
		\item[(c)] Suppose that $G/G_{\text{der}}$ is a $p'$-group and $|Z(G)|=|G^\ast/(G^\ast)_{\text{der}}|$. If $G^\ast$ contains a nontrivial semisimple element $s$ that satisfies conditions (\ref{condA}), (\ref{condB}), and (\ref{condC}), then $\OO[G_{\text{der}}/(G_{\text{der}}\cap Z(G))]$ possesses a non-principal real $2$-block.
\end{itemize}

\end{proposition}

\begin{proof}
	Let $1\neq s\in G^\ast$ be a semisimple element that satisfies condition (\ref{condA}) of Definition \ref{defn:conditions}. Then $s$ is a real $2'$-element of $G^\ast$ and $C_{G^\ast}(s)\subgp C_{\mathbf{G}^\ast}^\circ(s)$. Let $\chi_s\in\mathcal{E}(G,s)$ denote the corresponding semisimple character (see \ref{sschars} above). Since $s$ is real, Lemma \ref{lem:Jdecompprops} implies that $\chi_s$ is real-valued. Let $b$ denote the unique $2$-block of $\OO G$ such that $\chi_s\in\Irr(b)$. Then $b$ is real by Lemma \ref{lem:realblocksfingrps}. Since $\chi_s\in\mathcal{E}(G,s)\subseteq\mathcal{E}_2(G,s)$ we have $\Irr(b)\subseteq\mathcal{E}_2(G,s)$. Then, since $s\neq 1$, the remarks of \ref{blocksGF} above imply that $b$ is a non-principal $2$-block. Thus, (a) holds.
	
	Now assume that $G/G_{\text{der}}$ is a $p'$-group and that the element $s\in G^\ast$ of the previous paragraph is not $G^\ast$-conjugate to any element of the form $sz$, for $1\neq z\in Z(G^\ast)$. Note that $Z(G^\ast)=Z(\mathbf{G}^\ast)^{F^\ast}$ since $F$ is a Frobenius endomorphism (cf. \cite[Prop.~3.6.8]{Carter_1985}). Recall from \ref{Lusztigseries} the map $Z(G^\ast)\to\Lin(G)$, $z\mapsto\hat{z}$. As shown in \cite[Prop.~2.5.20]{Geck_2020}, this map is in fact an injective homomorphism with image equal to the unique largest $p'$-subgroup of the abelian group $\Lin(G)$. But by assumption $G/G_{\text{der}}\iso\Lin(G)$ is a $p'$-group; hence, the map $z\mapsto\hat{z}$ is an isomorphism between $Z(G^\ast)$ and the group of linear characters of $G$. 
	
	We claim that $\chi_s\cdot \lambda\neq\chi_s$ for all nontrivial linear characters $\lambda$ of $G$. Indeed, suppose $1\neq\lambda\in\Lin(G)$ is such that $\chi_s\cdot\lambda=\chi_s$. Let $z$ be the unique (necessarily nontrivial) element of $Z(G^\ast)$ such that $\lambda=\hat{z}$. Then $\chi_s\cdot\hat{z}=\chi_s$. But then, by the remarks of \ref{Lusztigseries}, the character $\chi_s$ belongs to the Lusztig series $\mathcal{E}(G,sz)$ as well as to $\mathcal{E}(G,s)$. It follows that $s$ is $G^\ast$-conjugate to $sz$, contrary to our assumption. Thus, the claim holds.
	
	For ease, set $N=G_{\text{der}}$ and $\psi_s=\Res_N^G\chi_s$. A well-known consequence of the claim proved in the previous paragraph is that the character $\psi_s$ is irreducible. Let $c$ be the unique $2$-block of $\OO N$ to which $\psi_s$ belongs. Since $\psi_s$ is real-valued, the block $c$ is real by Lemma \ref{lem:realblocksfingrps}. Suppose that $c=c_0$ is the principal block of $\OO N$. Then since the block $b$ of $\OO G$ covers $c$, Lemma \ref{lem:blockscoveringprincderived} implies that some linear character of $G$ belongs to $b$. Hence there exists $z\in Z(G^\ast)$ such that $\hat{z}\in\Irr(b)$. Then since $\Irr(b)\subseteq\mathcal{E}_2(G,s)$ there exists a $2$-element $x\in C_{G^\ast}(s)$ such that $\hat{z}\in\mathcal{E}(G,sx)$. Since $\hat{z}$ already belongs to $\mathcal{E}(G,z)$, it follows that $sx$ is $G^\ast$-conjugate to $z$. But $z$ is central, so in fact $sx=z$. Taking the $2'$-parts of either side of this equality, we find that $s=z_{2'}\in Z(G^\ast)$. Then $s$ is both a real and a central element of $G^\ast$, hence $s=s^{-1}$. This can only occur if $s=1$ or $s^2=1$, against our assumptions. Thus we find that $c$ is non-principal and (b) holds. 
	
	Continuing with all of the notation set above, we now establish (c). We still assume that $G/G_{\text{der}}$ is a $p'$-group, but now we also assume that $|Z(G)|=|G^\ast/(G^\ast)_{\text{der}}|$ and that the element $s\in G^\ast$ of the previous paragraphs belongs to $(G^\ast)_{\text{der}}$, i.e., satisfies condition (\ref{condC}) of Definition \ref{defn:conditions}. By \cite[Lem.~ 4.4(ii)]{Navarro_2013} the center $Z(G)$ is contained in the kernel of the character $\chi_s$. It follows that $N\cap Z(G)$ is a subgroup of $\ker\psi_s$. Let bars denote passage to the quotient $\bar{N}:=N/(N\cap Z(G))$. Then $\psi_s$ deflates to an irreducible real-valued character $\bar{\psi_s}$ of $\bar{N}$ defined by $\bar{\psi_s}(\bar{n})=\psi_s(n)$ for all $n\in N$. Let $d\in\Bl(\OO\bar{N})$ be such that $\bar{\psi_s}\in\Irr(d)$. Then the block $c$ of $\OO N$ dominates $d$. Note that $d$ is real by Lemma \ref{lem:realblocksfingrps}. The block $d$ must be non-principal, as otherwise $1_N\in\Irr(c)$, a contradiction. Therefore $\OO\bar{N}$ possesses a non-principal real $2$-block, completing the proof.
\end{proof}

\subsection{Linear and Unitary Groups}\label{subsec:linandunitary}

In this subsection we assume $n\geq 2$ and that $q$ is a positive integral power of an odd prime $p$. Let $\mathbf{G}=\GL_n(k)$, let $F_q$ denote the standard Frobenius endomorphism of $\mathbf{G}$ induced by the map $x\mapsto x^q$ on $k$, and let $\epsilon\in\set{\pm 1}$. Depending on the value of $\epsilon$, we let $F$ be one of two Frobenius endomorphisms of $\mathbf{G}$:
\begin{equation*}
	F=\begin{cases}
		F_q	&\text{if }\epsilon=1\\
		F_q\circ\gamma &\text{if }\epsilon=-1.
	\end{cases}
\end{equation*}
Here $\gamma:\mathbf{G}\to\mathbf{G}$ maps a matrix $g$ to $j_n(g^{\tr})^{-1}j_n$, where $j_n$ is the involutive permutation matrix
\begin{equation*}
	j_n:=\begin{pmatrix}
		0 & \cdots & 0 & 1\\
		\vdots & \iddots & \iddots & 0\\
		0 & 1 & \iddots & \vdots\\
		1 & 0 & \cdots & 0
	\end{pmatrix}.
\end{equation*}
Let
\begin{equation*}
	\GL_n(\epsilon q):=\mathbf{G}^F\qquad\text{and}\qquad \SL_n(\epsilon q):=(\mathbf{G}_{\text{der}})^F.
\end{equation*}
Note that $\GL_n(-q)=\GU_n(q)$ and $\SL_n(-q)=\SU_n(q)$ are, respectively, the finite general and special unitary groups. Recall that $\GL_n(-q)\subgp\GL_n(q^2)$ and $\SL_n(-q)=\GL_n(-q)\cap\SL_n(q^2)$. Finally, let
\begin{equation*}
	\PSL_n(\epsilon q):=\SL_n(\epsilon q)/Z(\SL_n(\epsilon q)).
\end{equation*}

\subsubsection{Additional Considerations for $\GL_n(\epsilon q)$}\label{maxmlysplittorusgl}\label{factsaboutgl}
	The pair $(\mathbf{G},F)$ is self-dual for  $\epsilon\in\{\pm1\}$ (cf. \cite[Example 1.5.21(a)]{Geck_2020}). Further, because $Z(\mathbf{G})$ is connected, the centralizer $C_{\mathbf{G}}(s)$ of any semisimple element $s\in \GL_n(\epsilon q)$ is connected. Since $q$ is odd, the derived subgroup of $\GL_n(\epsilon q)$ is equal to $\SL_n(\epsilon q)$. The center $Z(\GL_n(\epsilon q))$ and the quotient $\GL_n(\epsilon q)/\SL_n(\epsilon q)$ are both cyclic of order $q-\epsilon$.

	Let $\mathbf{T}_0$ denote the subgroup of diagonal matrices in $\mathbf{G}$. Then $\mathbf{T}_0$ is a maximally $F$-split torus. The Weyl group $\mathbf{W}=N_{\mathbf{G}}(\mathbf{T}_0)/\mathbf{T}_0$ is isomorphic to the subgroup of permutation matrices in $\mathbf{G}$, which in turn is isomorphic to the symmetric group $S_n$ of degree $n$.

\subsubsection{The Results for Type $\type{A}$}
\begin{lemma}\label{lem:typea}
	The group $\mathbf{G}^F=\GL_n(\epsilon q)$ possesses a nontrivial semisimple element $s$ satisfying conditions (\ref{condA}) and (\ref{condC}) of Definition \ref{defn:conditions} unless $(n,\epsilon q)\in\set{(2,\pm 3),(3,\pm 3)}.$
	
	Further, $\mathbf{G}^F=\GL_n(\epsilon q)$ possesses a nontrivial semisimple element $s$ satisfying conditions (\ref{condA}), (\ref{condB}), and (\ref{condC}) unless $(n,\epsilon q)\in\set{(2,\pm 3),(3,\pm 3),(3,-5),(3,7)}.$
\end{lemma}

\begin{proof}
	Assume that $(n,\epsilon q)\notin\set{(2,\pm 3),(3,\pm 3)}$. We must demonstrate that $\bG^F=\GL_n(\epsilon q)$ possesses a nontrivial semisimple element $s$ that is real, has odd order, and has determinant 1. Indeed, an element $s$ with these properties clearly satisfies condition (\ref{condC}) of Definition \ref{defn:conditions}, and satisfies condition (\ref{condA}) because the containment $C_{\mathbf{G}}(s)^F\subgp C_{\mathbf{G}}^\circ(s)=C_{\mathbf{G}}(s)$ is guaranteed (see \ref{factsaboutgl}). If in addition $(n,\epsilon q)\notin\set{(3,-5),(3,7)}$ then we must show that there exists such an element $s$ that also satisfies condition (\ref{condB}), i.e., such that $s$ is not $\GL_n(\epsilon q)$-conjugate to any element of the form $sz$, for $1\neq z\in Z(\GL_n(\epsilon q))$.
	
	First, suppose that there exists an odd number $m\neq 1$ that divides $q^2-1$.
    Let $\lambda\in\F_{q^2}^\times$ be a primitive $m$th root of unity. Then there exists a matrix $s\in \bG^F$ that is $\mathbf{G}$-conjugate to the diagonal matrix $d=\diag(\lambda,\lambda^{-1},1,\ldots,1)$. (Indeed, if $j$ denotes the image of the matrix $j_n$ in the Weyl group $\mathbf{W}$ and $w\in\mathbf{W}$ corresponds to the $2$-cycle $(1\,2)$, then $d\in\mathbf{T}_0[j^{\epsilon_0}]$ (see \ref{toritype}) if $m$ divides $q-\epsilon$ and $d\in\mathbf{T}_0[w\cdot j^{\epsilon_0}]$ if $m$ divides $q+\epsilon$, where $\epsilon_0=0$ if $\epsilon=1$ and $\epsilon_0=1$ if $\epsilon=-1$.)
	
	Then $s$ is a nontrivial semisimple $2'$-element of $\bG^F$ with $\det(s)=1$. Since $s$ is $\mathbf{G}$-conjugate to $s^{-1}$, Lemma \ref{lem:realelts} implies that $s$ is a real element of $\bG^F$. Therefore $s$ satisfies conditions (\ref{condA}) and (\ref{condC}) of Definition \ref{defn:conditions}. Assume that $(n,\epsilon q)\notin\set{(3,-5),(3,7)}$. We claim that $s$ also satisfies condition (\ref{condB}).
	
	Let $z\in Z(\bG^F)$ and suppose that $sz$ is $\bG^F$-conjugate to $s$. To prove the claim we must show that $z=1$. Since they are conjugate, note that $s$ and $sz$ have the same eigenvalues, including multiplicities. The matrix $s$ has (distinct) eigenvalues $\lambda$, $\lambda^{-1}$, and $1$ with corresponding eigenspaces of dimension $1$, $1$, and $n-2$, respectively. On the other hand, if $\zeta$ denotes the unique eigenvalue of $z$, then the eigenvalues of $sz$ are $\lambda\zeta$, $\lambda^{-1}\zeta$, and $\zeta$ with corresponding eigenspaces of dimension $1$, $1$, and $n-2$. Now, if $n\geq 4$ then comparing the eigenvalues with multiplicity $n-2>1$ shows that $\zeta=1$, and the claim follows. If $n=2$ then, comparing eigenvalues again, we must have $\lambda=\lambda\zeta$ or $\lambda=\lambda^{-1}\zeta$. But if $\lambda=\lambda^{-1}\zeta$ then the eigenvalue $\lambda^{-1}$ of $s$ must equal the eigenvalue $\lambda\zeta$ of $sz$, in which case $\lambda^{-2}=\zeta=\lambda^2$, contradicting the fact that the order of $\lambda$ is odd. So we must have $\lambda=\lambda\zeta$, hence $\zeta=1$ and the claim holds in this case. 
	
	To complete the proof of the claim it remains to consider the case when $n=3$. Suppose, in addition, that $m>3$. The eigenvalue $\lambda$ of $s$ must be equal to one of $\lambda\zeta$, $\lambda^{-1}\zeta$, or $\zeta$. Say $\lambda=\lambda^{-1}\zeta$. Then $\zeta=\lambda^2$ is an eigenvalue of $sz$, hence is an eigenvalue of $s$. Therefore, $\lambda^2$ is equal to $1$ or $\lambda^{-1}$. But neither equality is possible, since $\lambda$ has odd order $m>3$. Thus, $\lambda\neq\lambda^{-1}\zeta$. By a similar argument, $\lambda\neq\zeta$. We must have $\lambda=\lambda\zeta$, hence $\zeta=1$ and the claim holds. We are reduced to the case where $n=3$ and there does not exist an odd number $m>3$ that divides either $q-1$ or $q+1$. An elementary argument shows that this is only possible if $q\in\set{3,5,7}$. Our assumption on $(n,\epsilon q)$ then forces $\epsilon q=5$ or $\epsilon q=-7$. In either case, the center of $\GL_n(\epsilon q)$ is a $2$-group. Since $|s|=|sz|=|s|\cdot|z|$, this yields $z=1$.
	
	We may now assume that there does not exist an odd number $m>1$ that divides $q-1$ or $q+1$. Note that this can only occur if $q=3$. By our assumptions on $(n,\epsilon q)$, we must have $n\geq 4$. Arguing as before, there exists a nontrivial semisimple $s\in\GL_n(\epsilon\cdot 3)$ that is $\mathbf{G}$-conjugate to the diagonal matrix $\diag(\lambda,\lambda^2,\lambda^3,\lambda^4,1,\ldots,1)$, where $\lambda\in\F_{3^4}^{\times}$ is a primitive $5$th root of unity. By construction, the element $s$ has odd order and determinant 1. Since $s$ is $\mathbf{G}$-conjugate to $s^{-1}$, Lemma \ref{lem:realelts} implies that $s$ is a real element of $\GL_n(\epsilon\cdot 3)$. Finally, since $Z(\GL_n(\epsilon\cdot 3))$ is a $2$-group, if $z\in Z(\GL_n(\epsilon\cdot 3))$ is such that $s$ is conjugate to $sz$ then $|s|=|sz|=|s|\cdot|z|$, which implies that $z=1$. We conclude that the element $s$ satisfies all of the desired conditions, and the proof is complete.
\end{proof}

\begin{corollary}\label{cor:typea}
	Keep all the notation set above; in particular, $q$ is an integral power of an odd prime $p$, $n\geq 2$, and $\epsilon\in\set{\pm 1}$. Let $G\in \{\GL_n(\epsilon q), \SL_n(\epsilon q), \PSL_n(\epsilon q)\}$. Then $\OO G$ has no non-principal real $2$-block if and only if $(n,\epsilon q)\in\set{(2,\pm 3),(3,\pm 3)}$.
\end{corollary}

\begin{proof}
	The blocks of $\GL_n(\epsilon q)$ are described by Fong and Srinivasan in \cite{FS82}. While the results there are stated for odd $\ell$, analogous statements hold for the case $\ell=2$, as pointed out by Brou\'{e} in \cite{Broue_1986}. In particular, from this we see that the sets of $2$-blocks of $\OO[\GL_n(\epsilon q)]$ are in bijection with the conjugacy classes of semisimple $2'$-elements of $\GL_n(\epsilon q)$. In other words, if $s\in\GL_n(\epsilon q)$ is a semisimple $2'$-element, then $\mathcal{E}_2(\GL_n(\epsilon q),s)$ (as defined in \ref{blocksGF}) is equal to $\Irr(b_s)$ for a uniquely determined $2$-block $b_s$ of $\OO[\GL_n(\epsilon q)]$. From this fact and part (a) of Proposition \ref{prop:main}, we find that $\OO[\GL_n(\epsilon q)]$ has a non-principal real $2$-block if and only if there exists a nontrivial semisimple element $s\in\GL_n(\epsilon q)$ that is real and has odd order, i.e., satisfies condition (\ref{condA}) of Definition \ref{defn:conditions}. 
	
	Now, if $(n,\epsilon q)\in\set{(2,\pm 3),(3,\pm 3)}$ then $\GL_n(\epsilon q)$ does not possess a nontrivial semisimple real $2'$-element $s$. Indeed, $\GL_2(\pm 3)$ cannot possess such an element since the only primes dividing its order are $2$ and $3$. If $s\in\GL_3(3)$ is a nontrivial semisimple $2'$-element, then $s$ must be $\mathbf{G}$-conjugate to a diagonal matrix of the form $\diag(\lambda,\lambda^3,\lambda^9)$, where $\lambda\in\F_{3^3}^\times$ has order 13 and $\bG=\GL_n(k)$. Note that such an element $s$ is not real. Likewise, if $s\in\GL_3(-3)$ is a nontrivial semisimple $2'$-element, then $s$ must be $\mathbf{G}$-conjugate to a diagonal matrix $\diag(\lambda,\lambda^2,\lambda^4)$ where $\lambda\in\F_{3^6}^\times$ has order 7, which again cannot be real. It follows from the remarks of the previous paragraph that if $(n,\epsilon q)\in\set{(2,\pm 3),(3,\pm 3)}$, then $\OO[\GL_n(\epsilon q)]$ has no non-principal real $2$-block. On the other hand, if $(n,\epsilon q)\notin\set{(2,\pm 3),(3,\pm 3)}$ then Lemma \ref{lem:typea} implies that $\OO[\GL_n(\epsilon q)]$ has a non-principal real $2$-block. Thus, the statement for $\GL_n(\epsilon q)$ is established.

    Now, if $(n,\epsilon q)\notin\set{(2,\pm 3),(3,\pm 3),(3,-5),(3,7)}$ then both $\OO[\SL_n(\epsilon q)]$ and $\OO[\PSL_n(\epsilon q)]$ have non-principal real $2$-blocks by Lemma \ref{lem:typea} and Proposition \ref{prop:main}. Each of the group algebras $\OO[\SL_3(-5)]$, $\OO[\PSL_3(-5)]$, $\OO[\SL_3(7)]$, and $\OO[\PSL_3(7)]$ have an even number of $2$-blocks, which can be seen from GAP \cite{GAP}, so each must possess a non-principal real $2$-block. Finally, the group algebras $\OO[\SL_2(\pm 3)]$ and $\OO[\PSL_2(\pm 3)]$ have only one $2$-block (the principal block), and all non-principal $2$-blocks of the group algebras $\OO[\SL_3(\pm 3)]=\OO[\PSL_3(\pm 3)]$ have defect zero and consist of non-real irreducible characters. Hence, none of these group algebras possess a non-principal real $2$-block, completing the proof.
\end{proof}

\subsection{Symplectic and Odd-Dimensional Orthogonal Groups}

We continue to assume that $n\geq 2$ and that $q$ is a positive integral power of an odd prime $p$. Let
\begin{equation*}
	i_{2n}:=\begin{pmatrix}
		& j_n\\
		-j_n &
	\end{pmatrix}
\end{equation*}
where $j_n$ is the matrix defined in Subsection \ref{subsec:linandunitary}. Then
\begin{equation*}
	\Sp_{2n}(k)=\set{g\in\GL_{2n}(k)|g^{\tr}i_{2n}g=i_{2n}}
\end{equation*}
is the symplectic group of dimension $2n$ and
\begin{equation*}
	\SO_{2n+1}(k)=\set{g\in\SL_{2n+1}(k)|g^{\tr}j_{2n+1}g=j_{2n+1}}
\end{equation*}
is the special orthogonal group of dimension $2n+1$. 

By abuse of notation, we will use $F$ to denote the standard Frobenius endomorphism of either $\Sp_{2n}(k)$ or $\SO_{2n+1}(k)$, so that $\Sp_{2n}(k)^F=\Sp_{2n}(q)$ and $\SO_{2n+1}(k)^F=\SO_{2n+1}(q)$. Note that the pairs $(\Sp_{2n}(k),F)$ and $(\SO_{2n+1}(k),F)$ are dual to one another.

\subsubsection{Tori and Weyl Groups in Type $\type{B}$ and $\type{C}$}(cf. \cite[Prop.~11.4.3]{Carter_1985})\label{weyltypebandc}
	Let $\bG$ be $\Sp_{2n}(k)$ or $\SO_{2n+1}(k)$ and let $\mathbf{T}_0$ denote the intersection of $\bG$ with the diagonal torus of $\GL_{2n}(k)$, resp. $\GL_{2n+1}(k)$. Then $\mathbf{T}_0$ is a maximally $F$-split torus of $\bG$. Indeed, $\mathbf{T}_0$ is contained in the intersection of $\bG$ with the subgroup of upper-triangular matrices in $\GL_{2n}(k)$, resp. $\GL_{2n+1}(k)$, which is an $F$-stable Borel subgroup of $\bG$. The Weyl group $\mathbf{W}$ of $\bG$ is isomorphic to the wreath product $C_2\wr S_n$ defined with respect to the natural action of the symmetric group, hence is of order $2^n\cdot n!$.

\subsubsection{Useful Subgroups of $\Sp_{2n}(q)$ and $\SO_{2n+1}(q)$}\label{factsaboutsp}\label{factsaboutsoodd}
	As $n\geq 2$ and $q$ is odd, the symplectic group $\Sp_{2n}(q)$ is perfect and $Z(\Sp_{2n}(q))=\set{\pm\id_{2n}}$. The projective symplectic group is the quotient $$\PSp_{2n}(q):=\Sp_{2n}(q)/Z(\Sp_{2n}(q)).$$
	The center of $\SO_{2n+1}(q)$ is trivial, and the derived subgroup $$\Omega_{2n+1}(q):=\SO_{2n+1}(q)_{\text{der}}$$ is an index 2 subgroup of $\SO_{2n+1}(q)$.
	It will also be useful to note in what follows that we have injective group homomorphisms
	\[\begin{array}{rcl}
		\varphi:\GL_n(q)	&\into&\Sp_{2n}(q)\\
		g	&\mapsto&\begin{pmatrix}
			g & \\
			& j_n(g^{\tr})^{-1}j_n
		\end{pmatrix}
	\end{array}
	\quad\hbox{ and }\quad
	\begin{array}{rcl}
		\psi:\GL_n(q)	&\into&\SO_{2n+1}(q)\\
		g	&\mapsto&\begin{pmatrix}
			g & & \\
			& 1 &\\
			& & j_n(g^{\tr})^{-1}j_n
		\end{pmatrix},
	\end{array}\]
naturally allowing use to view $\GL_n(q)$ as a subgroup.

\subsubsection{The Results for Types $\type{B}$ and $\type{C}$}
Using the above and the results for $\GL_n(q)$ in Section \ref{subsec:linandunitary}, we are able to prove:

\begin{lemma}\label{lem:typec}
	The symplectic group $\Sp_{2n}(q)$, $n\geq 2$, possesses a nontrivial semisimple element $s$ that satisfies condition (\ref{condA}) of Definition \ref{defn:conditions}, and any such element automatically satisfies conditions (\ref{condB}) and (\ref{condC}).
\end{lemma}

\begin{proof}
	For ease, set $\mathbf{G}=\Sp_{2n}(k)$ and continue to let $F$ denote the standard Frobenius endomorphism of $\mathbf{G}$ such that $\mathbf{G}^F=\Sp_{2n}(q)$. First suppose $1\neq s\in\Sp_{2n}(q)$ is a semisimple element satisfying (\ref{condA}). Let $z\in Z(\Sp_{2n}(q))$ be such that $s$ and $sz$ are conjugate. As noted in Section \ref{factsaboutsp}, $Z(\Sp_{2n}(q))$ is cyclic of order 2. Since $s$ is a $2'$-element and $|s|=|sz|=|s|\cdot|z|$, we see $z=1$ and $s$ satisfies (\ref{condB}). Moreover, $s$ satisfies (\ref{condC}) because $\Sp_{2n}(q)_{\text{der}}=\Sp_{2n}(q)$. 
	
	To complete the proof, it remains to show that $\Sp_{2n}(q)$ possesses such an element $s$. Now, because $\mathbf{G}$ is simple of simply connected type, the centralizer $C_{\mathbf{G}}(s)$ of any semisimple element $s$ is connected. So we just need to show that there exists a nontrivial semisimple element $s\in\Sp_{2n}(q)$ that is real and has odd order. 
	
	If $(n,q)\notin\set{(2,3),(3,3)}$ then, by Lemma \ref{lem:typea}, there exists a nontrivial semisimple element $t\in\GL_n(q)$ that is real and has odd order. If we set $s=\varphi(t)\in\Sp_{2n}(q)$, where $\varphi$ is the injective homomorphism defined in  \ref{factsaboutsp}, then $s$ is a nontrivial semisimple element of $\Sp_{2n}(q)$ that satisfies (\ref{condA}). Thus, we are reduced to considering the groups $\Sp_4(3)$ and $\Sp_6(3)$. In fact, we only need to consider $\Sp_4(3)$, since $\Sp_6(3)$ contains a subgroup isomorphic to $\Sp_4(3)$.
	
	Now, the Weyl group of $\Sp_4(k)$ is isomorphic to $D_8$, the dihedral group of order 8. If $w\in\mathbf{W}$ is an element of order 4 then $\mathbf{T}_0[w]$, where $\mathbf{T}_0\subgp\Sp_4(k)$ is the maximally $F$-split torus of \ref{weyltypebandc}, contains the diagonal matrix $d=\diag(\lambda,\lambda^3,\lambda^2,\lambda^4)$, where $\lambda\in\F_{3^4}^\times$ is a primitive $5$th root of unity. Therefore, as discussed in \ref{toritype}, $\Sp_4(3)$ contains a nontrivial semisimple element $s$ that is $\mathbf{G}$-conjugate to $d$ and therefore has odd order. Finally, since $d$ is $\mathbf{G}$-conjugate to $d^{-1}$ (indeed, ${}^{i_4}d=d^{-1}$) we see that $s$ is $\mathbf{G}$-conjugate to $s^{-1}$, and Lemma \ref{lem:realelts} then implies that $s$ is a real element of $\Sp_4(3)$, as desired.
\end{proof}

\begin{lemma}\label{lem:typebodd}
	The odd-dimensional special orthogonal group $\SO_{2n+1}(q)$, $n\geq 2$, possesses a nontrivial semisimple element $s$ that satisfies conditions (\ref{condA}) and (\ref{condC}) of Definition \ref{defn:conditions}, and any such element automatically satisfies condition (\ref{condB}).
\end{lemma}

\begin{proof}
	Let $\mathbf{G}=\SO_{2n+1}(k)$ and continue to let $F$ denote the standard Frobenius endomorphism of $\mathbf{G}$ such that $\mathbf{G}^F=\SO_{2n+1}(q)$. We first note that any nontrivial semisimple element $s\in\SO_{2n+1}(q)$ satisfies (\ref{condB}) since, as noted in \ref{factsaboutsoodd}, the center of $\SO_{2n+1}(q)$ is trivial. 
	
	Recall that the spin group $\Spin_{2n+1}(k)$ is a semisimple algebraic group of simply connected type with the same root system as $\mathbf{G}$, and the kernel of the natural isogeny $\pi:\Spin_{2n+1}(k)\onto\mathbf{G}$ is cyclic of order 2. So if $s\in\mathbf{G}$ is a semisimple element with odd order, then $C_{\mathbf{G}}(s)$ is connected by Lemma \ref{lem:centconn}. Therefore, to complete the proof it suffices to demonstrate that $\SO_{2n+1}(q)$ possesses a nontrivial semisimple element $s$ that is real, has odd order, and belongs to $\SO_{2n+1}(q)_{\text{der}}=\Omega_{2n+1}(q)$.
	
	If $(n,q)\notin\set{(2,3),(3,3)}$, we argue as in the third paragraph of Lemma \ref{lem:typec}, but with $\psi$ in place of $\varphi$, to see that there
    is a nontrivial semisimple element with all of the desired properties. Thus, we are reduced to considering the groups $\SO_5(3)$ and $\SO_7(3)$, for which it suffices to consider $\SO_5(3)$, since again $\SO_7(3)$ contains a subgroup isomorphic to $\SO_5(3)$. 
	
	We proceed as in the proof of Lemma \ref{lem:typec}. The Weyl group $\mathbf{W}$ of $\mathbf{G}=\SO_5(k)$ is isomorphic to $D_8$, and if $w\in\mathbf{W}$ is an element of order 4 then $\mathbf{T}_0[w]$ contains the diagonal matrix $d=\diag(\lambda,\lambda^3,1,\lambda^2,\lambda^4)$, where $\lambda\in\F_{3^4}^\times$ is a primitive $5$th root of unity. The group $\SO_5(3)$ contains a nontrivial semisimple element $s$ that is $\mathbf{G}$-conjugate to $d$. Then $s$ is a $2'$-element, and since $d$ is $\mathbf{G}$-conjugate to $d^{-1}$ (indeed, ${}^{j_5}d=d^{-1}$) Lemma \ref{lem:realelts} implies that $s$ is a real element of $\SO_5(3)$. Finally, $s$ must belong to $\Omega_5(3)$ since, as noted in \ref{factsaboutsoodd}, $\Omega_5(3)$ has index $2$ in $\SO_5(3)$. This completes the proof.
\end{proof}

\begin{corollary}\label{cor:typebc}
	Assume that $n\geq 2$ and $q$ is an integral power of an odd prime $p$. Let $G\in\{\Sp_{2n}(q), \PSp_{2n}(q), \SO_{2n+1}(q), \Omega_{2n+1}(q)\}$. Then $\OO G$ possesses a non-principal real $2$-block.
\end{corollary}

\begin{proof}
	This follows from Proposition \ref{prop:main}, Lemmas \ref{lem:typec} and \ref{lem:typebodd}, and the facts discussed in \ref{factsaboutsp}.
\end{proof}

\subsection{Even-Dimensional Orthogonal Groups}\label{sec:evendimOrth}

In this subsection we assume $n\geq 4$. Continue to let $q$ denote a positive integral power of an odd prime $p$ and let $k=\bar{\F}_p$. Recall that the special orthogonal group of dimension $2n$ is the group
\begin{equation*}
	\SO_{2n}(k)=\set{g\in\SL_{2n}(k)|g^{\tr}j_{2n}g=j_{2n}}
\end{equation*}
where $j_{2n}$ is as in Subsection \ref{subsec:linandunitary}. 

Let $\mathbf{G}=\SO_{2n}(k)$ and let $F=F_q$ denote the standard Frobenius endomorphism of $\mathbf{G}$. The special orthogonal group of plus type is the group of fixed points $\SO_{2n}^+(q):=\SO_{2n}(k)^F$. Let $F'$ denote the Frobenius endomorphism of $\mathbf{G}$ defined by $F'(g)={}^x(F_q(g))$ for all $g\in\mathbf{G}$, where $x$ is the block diagonal matrix
\begin{equation*}
	x=\begin{pmatrix}
		\id_{n-1} & &\\
		& j_2 &\\
		& & \id_{n-1}
	\end{pmatrix}.
\end{equation*}
Then the special orthogonal group of minus type is the group of fixed points $\SO_{2n}^-(q):=\SO_{2n}(k)^{F'}$. The pairs $(\SO_{2n}(k),F)$ and $(\SO_{2n}(k),F')$ are each self-dual.

\subsubsection{Useful Subgroups of $\SO_{2n}^\pm(q)$}\label{factsaboutsoeven}\label{injmapsforevenso}
	Let $\epsilon$ denote one of the symbols $+$ or $-$. Then $\Omega_{2n}^\epsilon(q):=\SO_{2n}^\epsilon(q)_{\text{der}}$ and $\POmega_{2n}^\epsilon(q):=\Omega_{2n}^\epsilon(q)/Z(\Omega_{2n}^\epsilon(q))$. We have $|\SO_{2n}^\epsilon(q):\Omega_{2n}^\epsilon(q)|=2$ and 
$|Z(\Omega_{2n}^\epsilon(q))|\leq 2$. Moreover, $Z(\Omega_{2n}^\epsilon(q))=\Omega_{2n}^\epsilon(q)\cap Z(\SO_{2n}^\epsilon(q))$.
	We again have injective group homomorphisms
	\[\begin{array}{rcl}
		\varphi:\GL_n(q)	&\into&\SO_{2n}^+(q)\\
		g	&\mapsto&
		\begin{pmatrix}
			g & \\
			& j_n(g^{\tr})^{-1} j_n
		\end{pmatrix}
	\end{array}
	\quad\hbox{ and }\quad
	\begin{array}{rcl}
		\psi:\GL_{n-1}(q)	&\into&\SO_{2n}^-(q)\\
		g	&\mapsto&\begin{pmatrix}
			g & & \\
			& \id_2 &\\
			& & j_{n-1}(g^{\tr})^{-1}j_{n-1}
		\end{pmatrix}.
	\end{array}\]

\subsubsection{The Results for Types $\type{D}$ and $\tw{2}\type{D}$}

We may now proceed as in the cases of types $\type{B}$ and $\type{C}$ to complete the proof for the simple classical groups.
\begin{lemma}\label{lem:typed}
	The even-dimensional special orthogonal groups $\SO_{2n}^\epsilon(q)$, $n\geq 4$, each possess a nontrivial semisimple element $s$ that satisfies conditions (\ref{condA}) and (\ref{condC}) of Definition \ref{defn:conditions}, and any such element automatically satisfies condition (\ref{condB}).
\end{lemma}

\begin{proof}
	The proof is similar to the proof of Lemma \ref{lem:typebodd}. First note that any semisimple element $s\in\SO_{2n}^{\epsilon}(q)$ of odd order satisfies condition (\ref{condB}) since $Z(\SO_{2n}^\epsilon(q))$ is a $2$-group. Note also that such an element $s$ has a connected centralizer in $\SO_{2n}(k)$ by Lemma \ref{lem:centconn}. Now, using Lemma \ref{lem:typea} and the injective homomorphisms of \ref{injmapsforevenso}, one can produce a nontrivial semisimple element $s\in\SO_{2n}^{\epsilon}(q)$, except possibly when $\epsilon=-$, $n=4$, and $q=3$. In the latter case, we can produce a nontrivial semisimple element $s\in\SO_8^-(3)$ satisfying the desired conditions by following the same steps used in the final part of the proof of Lemma \ref{lem:typebodd}.
\end{proof}

\begin{corollary}\label{cor:typed}
	Assume that $\epsilon\in\set{+,-}$, $n\geq 4$, and $q$ is an integral power of an odd prime $p$. Let $G\in\{ \SO_{2n}^\epsilon(q), \Omega_{2n}^\epsilon(q), \POmega_{2n}^\epsilon(q) \}$. Then $\OO G$ possesses a non-principal real $2$-block.
\end{corollary}

\begin{proof}
	This follows from Proposition \ref{prop:main}, Lemma \ref{lem:typed}, and the facts discussed in \ref{factsaboutsoeven}.
\end{proof}

\subsection{Exceptional Groups}

We continue to assume that $q$ is a positive integral power of an odd prime $p$ and now turn our attention at the exceptional groups of Lie type. We first consider the groups $\type{G}_2(q)$ and $\tw{3}\type{D}_4(q)$.

\begin{lemma}\label{lem:g23d4}
    Let $\bG$  be a simple algebraic group and $F$ a Frobenius endomorphism such that $\bG^F\in\{\type{G}_2(q), \tw{3}\type{D}_4(q)\}$. Then $\OO[\bG^F]$ possesses a non-principal real $2$-block.
\end{lemma}

\begin{proof}
First suppose that $\bG^F=\type{G}_2(q)$. The character table of $\bG^F$ has been determined in \cite{changree, enomoto}. From this, we see  that $\type{G}_2(q)$ has precisely 2 non-real irreducible characters.
    Further, the results of \cite{Hiss_1992} show that $\OO[\mathbf{G}^F]$ has more than three $2$-blocks. We conclude that $\OO[\mathbf{G}^F]$ must possess a non-principal real $2$-block.

    Next, consider the case $\bG^F=\tw{3}\type{D}_4(q)$. In this case, the character table is determined in \cite[Sec.~4]{Deriziotis_1987}. From this, we see that every irreducible character of $\mathbf{G}^F$ is real. Further, the proof of \cite[Cor.~5.1]{Deriziotis_1987} shows that $\OO[\mathbf{G}^F]$ has $\frac{1}{4}(q^4-q^2)$ defect-zero $2$-blocks, each of which must be non-principal.
\end{proof}

In what follows, we will write $\type{E}_6(\epsilon q)_{\text{sc}}$ for the group $\type{E}_6(q)_{\text{sc}}$ when $\epsilon=+$ and $\tw{2}\type{E}_6(q)_{\text{sc}}$ when $\epsilon=-1$. These are the groups of Lie type of the form $\bG^F$, where $\bG$ is a simple simply connected algebraic group with root datum of Cartan type $\type{E}_6$ and $F\colon \bG\rightarrow\bG$ is a standard, respectively twisted, Frobenius endomorphism. Similarly, the group $\type{E}_6(\epsilon q)_{\text{ad}}$ will denote the group obtained from the simple algebraic group of adjoint type $\type{E}_6$, and  $\type{E}_7(q)_{\text{sc}}$, resp.  $\type{E}_7(q)_{\text{ad}}$, will denote the group of Lie type $\bG^F$, where $\bG$ is a simple simply connected algebraic group, respectively simple algbraic group of adjoint type, with root datum of Cartan type $\type{E}_7$ and $F\colon \bG\rightarrow\bG$ is a Frobenius endomorphism. 

\begin{lemma}\label{lem:f4e6}
Let $\bG$ be a simple algebraic group and $F$ a Frobenius endomorphism such that $\bG^F\in\{\type{F}_4(q), \type{E}_{6}(\epsilon q)_{\text{sc}}\}$.  Then $\OO[\bG^F]$ and $\OO[\bG^F/Z(\bG^F)]$ each possess a non-principal real $2$-block.
\end{lemma}
\begin{proof}
First assume that $\bG^F=\type{F}_4(q)$. Note that in this case, $(\bG, F)$ is self-dual,  $Z(\bG)$ is trivial, and $\bG^F=\bG^F/Z(\bG^F)$. Further, $\bG^F$ contains subgroups isomorphic to $\Sp_6(q)$. Then by Lemma \ref{lem:typec}, there is a semisimple $2'$-element $s\in (\bG^\ast)^{F^\ast}$ satisfying condition (\ref{condA}) of Definition \ref{defn:conditions}. (Note that $C_{\bG^\ast}(s)$ is necessarily connected since $Z(\bG)$ is.) Then the statement follows from Proposition \ref{prop:main}(a).

Now suppose $G=\bG^F=\type{E}_{6}(\epsilon q)_{\text{sc}}$, and write $S:=\bG^F/Z(\bG^F)$. Further,  let $\mathbf{H}=\type{E}_{6,ad}$ be the corresponding simple algebraic group of adjoint type, with $H=\mathbf{H}^F=\type{E}_{6}(\epsilon q)_{\text{ad}}$. Then $H^\ast=G$ contains a subgroup isomorphic to $\type{F}_4(q)$, so from above there exists a real semisimple $2'$-element of $H^\ast$. Further, $C_{\mathbf{H}^\ast}(s)$ is connected since $\mathbf{H}$ is of adjoint type, so $H$ has a non-principal real $2$-block $B$ by Proposition \ref{prop:main}(a). 
    
    Now, the simple group $S$ is isomorphic to a normal subgroup of $H$ with index dividing 3, and we identify $S$ with this subgroup. Here $B$ cannot lie above the principal block $B_0(S)$ of $S$, as then by Lemma \ref{lem:blockscoveringprincderived}, $\Irr(B)=\Irr(B_0(H))\cdot \beta$ for some $\beta\in\Irr(H/S)$ of order $3$, where $B_0(H)$ is the principal block of $H$, contradicting that $B$ is real. Then $B$ lies above either one or three non-principal blocks of $S$, and since $B$ must also lie above the complex conjugates, at least one of these blocks of $S$ is real. Then $\OO[S]$ contains a real non-principal $2$-block. But this must be dominated by a real non-principal $2$-block of $\OO[G]$ using Lemma \ref{lem:domination}, completing the proof.
\end{proof}

\begin{lemma}\label{lem:e7e8}
	Let $\mathbf{G}$ be a simply connected simple algebraic group 
    and let $F$ be a Frobenius endomorphism of $\mathbf{G}$ such that $\mathbf{G}^F\in\{\type{E}_7(q)_{\text{sc}}, \type{E}_8(q)\}$. Then $\OO[\bG^F]$ and $\OO[\bG^F/Z(\bG^F)]$ each possess a non-principal real $2$-block.
\end{lemma}

\begin{proof}
Let $(\mathbf{G}^\ast,F^\ast)$ be dual to $(\mathbf{G},F)$. Then $\mathbf{G}^\ast$ is a simple algebraic group of adjoint type with root datum of Cartan type $\type{E}_7$, respectively $\type{E}_8$, and $\mathbf{G}^{\ast F^\ast}=\type{E}_7(q)_{\text{ad}}$, respectively $\type{E}_8(q)$. (Note that $\type{E}_8(q)$ is self-dual.) By \cite[Thm.~2.3.1]{Singh_2008}, every semisimple element in $\mathbf{G}^{\ast F^\ast}$ is real. Therefore, there must exist a nontrivial semisimple element $s\in\mathbf{G}^{\ast F^\ast}$ that is real and has odd order. Now, the kernel of the natural isogeny $\pi:\mathbf{G}\onto\mathbf{G}^\ast$ is cyclic of order dividing 2 (c.f. \cite[Rem.~1.5.13]{Geck_2020}). Hence, Lemma \ref{lem:centconn} implies that the centralizer $C_{\mathbf{G}^\ast}(s)$ is connected. We see then that the element $s$ satisfies condition (\ref{condA}) of Definition \ref{defn:conditions}. Since $Z(\mathbf{G}^{\ast F^\ast})$ is trivial and since the derived subgroup of $\type{E}_7(q)_{\text{ad}}$ has index 2 in $\type{E}_7(q)_{\text{ad}}$ and $\type{E}_8(q)$ is simple, the element $s$ clearly satisfies conditions (\ref{condB}) and (\ref{condC}). Noting that $\type{E}_7(q)_{\text{sc}}$ is perfect and $|Z(\type{E}_7(q)_{\text{sc}})|=2$, Proposition \ref{prop:main} allows us to conclude that both $\OO[\bG^F]$ and $\OO[\bG^F/Z(\bG^F)]$ possess a non-principal real $2$-block.
\end{proof}

\section{Proof of the Main Results }\label{sec:mainproofs}

Finally, we end with the proofs of Theorem \ref{thm:real2blocks} and Corollary \ref{cor:quadraticselfdual}.

\begin{proof}[Proof of Theorem \ref{thm:real2blocks}]
If $G$ is simple, we may assume by Lemma \ref{lem:sporalt} that $G=\bG^F/Z(\bG^F)$ for $\bG$ a simple, simply connected algebraic group and $F\colon \bG\rightarrow\bG$ a Steinberg endomorphism. Then the result follows from Lemmas \ref{lem:defcharRee}, \ref{lem:g23d4}, \ref{lem:f4e6}, and \ref{lem:e7e8} together with Corollaries \ref{cor:typea}, \ref{cor:typebc}, and \ref{cor:typed}.

So, assume that $Z(G)\neq 1$.
If $\OO[G/Z(G)]$ has a non-principal real $2$-block $b$, then so does $\OO[G]$ using Lemma \ref{lem:domination}, by considering a block dominating $b$. It suffices to see that the covering groups of the exceptions listed in the theorem do not have non-principal real $2$-blocks. 

The groups $\PSL_3(3)$, $\PSU_3(3)$, $M_{11}$, $M_{23}$, and $M_{24}$ each have trivial Schur multiplier, so it further suffices to check the covering groups of $M_{22}$, which has a cyclic Schur multiplier of size $12$. Again using \cite{GAP}, we see that these covers have no non-principal real $2$-block, and this completes the proof. 
\end{proof}
\begin{proof}[Proof of Corollary \ref{cor:quadraticselfdual}]
First, if $G$ is not one of the simple groups excluded in Theorem \ref{thm:real2blocks}, then $G$ has a real non-principal $2$-block $B$. By \cite[Prop 1.4]{GowWillems93}, $B$ then contains a self-dual Brauer character $\psi$, which by \cite[Lem.~17]{GowMurray} must be of quadratic type. 

Now suppose that $G$ is one of the simple groups excluded in Theorem \ref{thm:real2blocks}. The statement for sporadic groups is shown in \cite[Sec.~5]{GowMurray}. So, we are left to consider the groups $\PSL_3(\epsilon 3)$. Here we argue similarly to the case of sporadic groups in \cite[Sec.~5]{GowMurray} - namely, we see in GAP \cite{GAP} that there is an orthogonal character (of degree 12 when $\epsilon =1$, respectively degree 14 when $\epsilon=-1$) that restricts to a self-dual irreducible Brauer character when considered on $2$-regular elements.
\end{proof}


\end{document}